\documentclass[11pt]{amsart}
\usepackage{geometry}                % See geometry.pdf to learn the layout options. There are lots.
\usepackage{graphicx}
\usepackage{color}

\usepackage{amsthm}
\usepackage{amsmath}
\usepackage{amsfonts}
\usepackage{enumerate}
\usepackage{amssymb}
\usepackage{epstopdf}
\usepackage{hyperref}

\newtheorem{theorem}{Theorem}[section]
\newtheorem{proposition}[theorem]{Proposition}
\newtheorem{lemma}[theorem]{Lemma}
\newtheorem{corollary}[theorem]{Corollary}

\DeclareMathOperator{\W}{W}
\DeclareMathOperator{\FH}{FH}
\DeclareMathOperator{\V}{V}
\DeclareMathOperator{\Vol}{Vol}
\DeclareMathOperator{\Area}{Area}
\DeclareMathOperator{\Ricci}{Ricci}

\DeclareMathOperator{\sys}{sys}
\DeclareMathOperator{\MCV}{MCV}

\theoremstyle{definition}
\newtheorem{definition}[theorem]{Definition}
\newtheorem{notation}[theorem]{Notation}
\newtheorem{remark}[theorem]{Remark}

\title{Width of conformal metrics}
\author{Parker Glynn-Adey}
\author{Yevgeny Liokumovich}

\begin{document}

\title{Width, Ricci curvature, and minimal hypersurfaces}

\newcommand{\R}{\mathbb{R}}
\newcommand{\RP}{\mathbb{RP}}
\newcommand{\Z}{\mathbb{Z}}

\maketitle

\begin{abstract}
    Let $(M,g_0)$ be a closed
    Riemannian manifold of dimension $n$, for $3 \leq n \leq 7$, and
    non-negative Ricci curvature.
    Let $g = \phi^2 g_0$ be a metric in the
    conformal class of $g_0$.
    We show that there exists a smooth closed
    embedded minimal hypersurface in $(M,g)$
    of volume bounded by $C(n) V^{\frac{n-1}{n}}$,
    where $V$ is the total volume of $(M,g)$. When $Ric(M,g_0) \geq -(n-1)$
    we obtain a similar bound with constant $C$
    depending only on $n$ and the volume of $(M,g_0)$.
     
    Our second result concerns manifolds $(M,g)$
    of positive Ricci curvature and dimension at most seven.
    We obtain an effective version of a theorem
    of F. C. Marques and A. Neves on the existence of infinitely
    many minimal hypersurfaces on $(M,g)$ .
    We show that for any such
    manifold there exists $k$ minimal hypersurfaces
    of volume at most $C_n  V \left( \sys_{n-1}(M)\right)^{-\frac{1}{n-1}} k ^ {\frac{1}{n-1}}$,
    where $V$ denotes the volume of $(M,g_0)$ and $\sys_{n-1}(M)$
    is the smallest volume of a non-trivial minimal
    hypersurface.
\end{abstract}

\section{Introduction}
\label{sec:Introduction}

\subsection{Results}

In \cite{Pitts1981} Pitts proved existence of
a smooth closed embedded minimal hypersurface
in any closed Riemannian manifold $M$
of dimension $n$, for $3 \leq n \leq 5$.
This result was extended to manifolds of dimension
$n \leq 7$ by Schoen and Simon \cite{SchoenSimon}.
Our first main result is a bound on the volume
of this hypersurface for certain conformal classes
of Riemannian metrics.

%\begin{definition}\label{definition: covering constants}
%    Let $M_0$ be an $n$-dimensional closed Riemannian
%    manifold. Define $\tau(M_0)$ and $\nu(M_0)$ as
%    follows: $\tau(M_0)$ is the least number such that
%    any annulus $B(x,2r) \setminus B(x,r)$ in $M_0$
%    can be covered by $\tau(M_0)$ balls of radius $r$.
%    We let $\nu(M_0)$ be the least constant such that
%    $\Vol (B(x,r)) \leq \nu(M_0) r^n$ for all $r > 0$
%    and all $x \in M_0$.
%\end{definition}

\begin{theorem} \label{hypersurface}
    Suppose $M_0$ is a closed Riemannian manifold
    of dimension $n$, for $3 \leq n \leq 7$.
    If $M$ is in the conformal class of $M_0$
    then $M$ contains a smooth
    closed embedded minimal hypersurface $\Sigma$
    with volume bounded above by
    $C(M_0) \Vol(M)^{\frac{n-1}{n}}$.
    When $\Ricci(M_0) \geq 0$ the constant $C(M_0)$
    is an absolute constant that depends only on $n$.
    In general, for $M_0$ with $\Ricci(M_0) \geq -(n-1)a^2$ we can take 
    $C(M_0)=C(n) \max \{1,a \Vol(M_0)^ {\frac{1}{n}}\}$.
    
    If $n>7$ the same upper bound will hold
    for the $(n-1)$-volume of a closed minimal 
    hypersurface with singularities of dimension
    at most $n-8$.
\end{theorem}

Theorem \ref{hypersurface}
follows from a bound on the width of $M$.
One can find background and
many results about widths of manifolds in \cite[App.1F]{Gromov1983}, \cite{Guth2007},
\cite{BalacheffSabourau2010}, \cite{MarquesNeves2013}.
Informally, the width $\W(M)$ of a manifold $M$
is the smallest number such that every sweep-out
of $M$ by hypersurfaces contains a hypersurface
of volume at least $\W(M)$. We give a precise definition of width in
Section \ref{sec: definitions}.

To state our bound on the width of manifolds
it will be convenient to define a conformal 
invariant called the min-conformal volume. 
This invariant was recently introduced in a
work of Hassannezhad \cite{Hassannezhad}.

\begin{definition}\label{MCV}
Let $M$ be a compact Riemannian manifold.
Define the min-conformal volume of $M$ to be: $\MCV(M) = \inf \{ \Vol(M') \}$, where the infimum
is taken over all manifolds $M'$ in the conformal class of $M$ with 
$\Ricci (M') \geq -(n-1)$.
\end{definition}

\begin{theorem} \label{thm: main}
    Let $M$ be a closed Riemannian manifold of dimension $n$
    then 
    \[
    \W(M) \leq
    C(n) \max \{1,\MCV(M)^{\frac{1}{n}}\} \Vol(M)^{\frac{n-1}{n}}
    \]
\end{theorem}

\begin{corollary}
    Let $M$ be a closed Riemannian manifold of dimension $n$, $a\geq 0$
    and suppose that $\Ricci(M)\geq -(n-1)a^2$. Then 
    \[
    \W(M) \leq
    C(n) \max \{1,\ a \Vol(M)^{\frac{1}{n}}\} \Vol(M)^{\frac{n-1}{n}}
    \]
\end{corollary}

More generally, Theorem \ref{thm: main} holds
if we replace the min-conformal volume in the estimate
by the infimum of $\{d \Vol(M')\}$, where $d$ is any positive integer,
such that $M$ admits a degree $d$ conformal branched covering 
onto a manifold $M'$ with $\Ricci (M') \geq -(n-1)$. 

The conformal invariant $\MCV$ is somewhat reminiscent of 
(but different from) the
conformal volume studied by Li and Yau in \cite{LiYau1982}.
$\MCV(M)$ can be computed explicitly when $M$ is a Riemannian surface of genus $g$.
By the uniformization theorem $M$ is conformally equivalent to a surface $M_0$
of constant curvature $1$ (if $g=0$), $0$ (if $g=1$) or $-1$ (if $g \geq 2$). When the genus is $0$ or $1$ it follows that
$\MCV(M)=0$. 
When $g\geq 2$ 
%we may apply Gauss-Bonnet and conclude $\MCV(M)= 4 \pi (g-1)$.
and the Gaussian curvature 
of a surface $M'$ satisfies $K \geq -1$ then by Gauss-Bonnet theorem $\Area(M') \geq 4 \pi(g-1)$ with equality holding exactly when 
$K=-1$ everywhere. We conclude that $\MCV(M)=4 \pi(g-1)$

%the genus is at least $2$ and the Gaussian curvature 
%of a surface $M'$ satisfies $K \geq -1$ then by Gauss-Bonnet $\Area(M') \geq 4 \pi(g-1)$ with equality holding exactly when 
%$K=-1$ everywhere. We conclude that $\MCV(M)=4 \pi(g-1)$

Theorem \ref{thm: main} then implies that for any surface $M$ of genus $g$ we have 
\[ \W(M) \leq C \sqrt{(g+1) \Area(M)}\]
This result was previously obtained by Balacheff and Sabourau 
in \cite{BalacheffSabourau2010} with 
constant $C=10^8$. Using a slightly modified version
of our proof and invoking the Riemann-Roch theorem we can get a somewhat better constant
for orientable surfaces. 

\begin{theorem} \label{BS}
Any closed orientable Riemannian manifold $S_g$ of dimension $2$ and genus $g$ satisfies
\[ \W(S_g) \leq
            220 \sqrt{(g+1) \Area(S_g)} \]
\end{theorem}

Upper bounds on the higher parametric versions of width
$\W^k(M)$ for all Riemannian surfaces were recently obtained by 
the second author \cite{Liokumovich2014}.

In \cite{Brooks1986} Brooks constructed
hyperbolic surfaces of large genus
and Cheeger constant bounded away from zero.
These surfaces have width $\W(M)$
bounded below by $c  \sqrt{g \Area(M)}$
for some constant $c>0$.
Hence, the inequality in Theorem \ref{thm: main}
is optimal up to the value of the constant $C(n)$.

It follows from the works of Almgren \cite{Almgren1965},
Pitts \cite{Pitts1981}, and Schoen and Simon \cite{SchoenSimon} that
estimates on width yield upper bounds on the volume of smooth
embedded minimal hypersurfaces in manifolds of dimension less than
or equal to $7$. In higher dimensions, we
obtain bounds on the volume of stationary integral varifolds,
which are smooth hypersurfaces everywhere except possibly for a set of
Hausdorff dimension at most $n-8$.

It is possible to obtain more minimal hypersurfaces
if one considers parametric families of sweep-outs.
In Section \ref{sec: definitions} we define families
of hypersurfaces that correspond to cohomology
classes of mod 2 $(n-1)$-cycles on $M$.
To each such $p$-dimensional family we
assign the corresponding min-max quantity $\W^p(M)$.
Let $S^n$ be the round unit $n$-sphere.
In \cite[4.2.B]{Gromov1988} Gromov showed that
there are constants $0<c(n)<C(n)$ so that $\W^p(S^n)$ satisfies:

\[ c(n) p^{\frac{1}{n}} \Vol(S^n)^{\frac{n-1}{n}}
\leq \W^p(S^n) \leq C(n) p^{\frac{1}{n}} \Vol(S^n)^{\frac{n-1}{n}}  \]

Guth \cite{Guth2009} derived similar bounds for min-max
quantities corresponding to the Steenrod algebra generated by the fundamental class
$\lambda$.
%($\lambda$ is the cohomology class representing
%the sweep-out of $M$, see Section \ref{sec: definitions}).
Marques and Neves~\cite{MarquesNeves2013}, building on the work
of Gromov and Guth, proved existence of infinitely many minimal hypersurfaces
on a manifold $M$ of dimension $n$, for $3 \leq n \leq 7$,
under the assumption that $M$ has positive Ricci curvature.

In Section~\ref{sec:Volumes of hypersurfaces} we show that if $M$ has
non-negative Ricci curvature
then $\W^p(M) \leq C(n) p^{\frac{1}{n}} \Vol(M)^{\frac{n-1}{n}}$.
We use this bound to derive an effective version
of the theorem of Marques and Neves.
Let $\sys_{n-1}(M)$ be the infimum of volumes of 
smooth closed embedded minimal hypersurfaces in $M$.

\begin{theorem} \label{hypersurfaces}
    Suppose $M$ is a closed Riemannian manifold
    of dimension $n$, $3 \leq n \leq 7$, and positive
    Ricci curvature.
    For every $k$ there exists $k$ smooth
    closed embedded minimal hypersurfaces
    of volume bounded above by $C(n)
    k ^ {\frac{1}{n-1}} \Vol(M) \left( \sys_{n-1}(M) \right)^{- \frac{1}{n-1}}$,
    where $C(n)$ depends only on $n$.
\end{theorem}

\subsection{Previous work.}
The main estimates in our paper were motivated by
similar estimates on the
spectrum of the Laplace operator on Riemannian manifolds.

Let $M$ be a closed Riemannian manifold in
the conformal class of $M_0$. In \cite{Korevaar1993}
Korevaar constructed 
a decomposition of $M$ into annuli (and other regions) which
measures the `volume concentration' of the metric $M$ with
respect to the base metric of $M_0$. This annular
decomposition is then used to estimate Rayleigh quotients,
thus bounding the spectrum of the Laplacian of $M$.
Korevaar's method was further developed by
Grigor'yan-Yau in \cite{Grigoryan1999} and
Grigor'yan-Netrusov-Yau in \cite{Grigoryan2004}
to obtain upper bounds on the eigenvalues of elliptic operators
on various metric spaces.
In \cite{Gromov1993} Gromov used a different approach
(based on Kato's inequality) to obtain upper bounds for
the spectrum of the Laplacian on K{\"a}hler manifolds \cite{Gromov1993}.

In \cite{Hassannezhad} Hassannezhad, combining methods 
of \cite{ColboisMaerten} and \cite{Grigoryan2004}, obtained
upper bounds for eigenvalues of the  Laplacian
in terms of the conformal invariant $\MCV$ (see Definition \ref{MCV})
and the volume of the manifold.

As suggested by Gromov in \cite{Gromov1988}
the problem of bounding width $\W(M)$ and its parametric
version $\W^k$
can be thought of as a nonlinear analogue of finding the spectrum of the Laplacian on $M$.
In this paper we were guided by this analogy.

In \cite{Guth2007} Guth constructed sweep-outs
of open subsets of Euclidean space by $k$-cycles
of controlled volume for all $k$, $1 \leq k \leq n-1$.
In particular, he proved the following

\begin{theorem}[Guth~\cite{Guth2007}] \label{guth}
    For every open subset $U \subset \mathbb{R}^n$,
$W(U) \leq C(n) \Vol(U) ^ {\frac{n-1}{n}}$.
\end{theorem}

In dimension $n \geq 3$ one can find a family of
parallel hyperplanes in $\R^n$ yielding the desired sweep-out.
This follows from the work of Falconer \cite{Falconer1980}
on the $(n,k)$-Besicovitch conjecture.
In dimension $2$, however, it may happen that any slicing
of $U$ by parallel lines contains an arbitrarily large segment.
To surpass problems of this kind
Guth developed a method of
sweeping out regions by `bending planes' around
the skeleton of the unit lattice in $\mathbb{R}^n$.
This method was developed further in \cite{Guth2009}
to bound higher parametric versions of width.

It follows from the work of Burago and Ivanov \cite{BuragoIvanov}
(see also Appendix 5 of \cite{Guth2007})
that on any manifold of dimension
greater than two there exists a Riemannian metric
of small volume and arbitrarily large $(n-1)$-width.
Our results show that this does not happen
for certain conformal classes of manifolds.
In particular, this does not happen in the
presence of a curvature bound.

In \cite{NabutovskyRotman2004} Nabutovsky and Rotman
showed that any closed Riemannian manifold
possesses a stationary $1$-cycle of mass
bounded by $C(n) \Vol(M) ^{\frac{1}{n}}$.
One wonders if this result
can be generalized to the case of minimal surfaces
on Riemannian manifolds.

Some results in this direction were obtained by A. Nabutovsky and R. Rotman in  \cite{NabutovskyRotman2006}
where they bounded volumes of minimal surfaces on Riemannian
manifolds in terms of homological filling functions of $M$.
The $k$-th homological filling function $\FH_k: \R_+ \rightarrow \R_+$
is defined as the smallest number $\FH_k(x)$, such that every
$k$-cycle of mass at most $x$ can be filled by a
$(k+1)$-chain of mass at most $\FH_k(x)+\epsilon$.

\begin{theorem}[Nabutovsky-Rotman~\cite{NabutovskyRotman2004}] \label{NR}
Let $M$ be a closed Riemannian manifold of dimension $n$, for $3 \leq n \leq 7$,
such that the first $n-1$ homology groups are trivial,
$H_1(M) = ... =H_{n-1}(M)=0$. There exists a
smooth, closed, embedded minimal hypersurface of volume
bounded by $C(n) \FH_{n-1}(C(n) \FH_{n-2}( \dots \FH_2 (C(n) \Vol(M)^ {\frac{1}{n}})\dots))$.
\end{theorem}

Their proof uses a combination of Almgren-Pitts min-max method
and other techniques. In particular, a bound
on the width of $M$ in terms of homological filling
functions does not follow from their argument.
It would be interesting to know whether such a bound exists.
It is also interesting to know whether
homological filling
functions can be controlled in terms of Ricci curvature
of $M$.

Other important results are contained in a paper
of Marques and Neves \cite{MarquesNeves2012}
where, among other things, they prove a \emph{sharp}
upper bound on $W(M)$, when $M$ is a Riemannian 3-sphere
with $\Ricci>0$ and scalar curvature $R \geq 6$.

\subsection{Plan of the Paper}

The structure of the proof of Theorem \ref{thm: main} is as follows:
To construct
a sweep-out of $M$, we subdivide $M$ repeatedly, using
an isoperimetric inequality adapted to our context.
Once we have subdivided $M$ into a collection of small
volume open subsets, we construct a sweep-out of each
small volume piece using the fact that at small scales
$M$ is locally Euclidean. We then assemble these local
sweep-outs into a global sweep-out of $M$.

In Section~\ref{sec: definitions} we define what it means
for a family of $(n-1)$-cycles to sweep-out $M$. We define the
width $\W(M)$ and its higher parametric version
$\W^k(M)$. We also prove Proposition \ref{decomposition},
which gives us control of the width of $M$
in terms of widths of its open subsets.

In Section~\ref{sec:Isoperimetric inequality} we 
use an idea of Colbois and Maerten from \cite{ColboisMaerten} 
together with the length-area method to prove
an isoperimetric inequality
(Theorem~\ref{thm: isoperimetric})
which allows us to partition any open set in
$M$ in two parts with both parts satisfying a lower
volume bound. The subdividing surface satisfies an upper
bound on area which depends on the volume of the open set.
%The lower
%bound on volume of the two parts ensures that after repeatedly subdividing,
%using Theorem~\ref{thm: isoperimetric}, we will have
%many subsets of small volume.

In Section~\ref{sec: small submanifolds} we estimate
the width of small volume submanifold $M' \subset M$
in terms of $(n-1)$-volume of its boundary.
The proof proceeds by covering $M'$ with
small balls, which are $(1+\epsilon_0)$-bilipschitz
diffeomorphic to balls in Euclidean space.
We construct a sequence of nested open subsets
$U_i$ of $M'$ with volumes tending to $0$,
such that the difference $U_{i} \setminus U_{i+1}$
is contained in a small ball.
Since the ball is almost Euclidean,
we can sweep out $U_{i} \setminus U_{i+1}$ by cycles of controlled volume.
We then use Proposition \ref{decomposition}
to assemble a sweep-out of $M'$.

%We use Theorem \ref{guth} to construct a sweep-out in each small ball
%and then assemble these sweep-outs to
%a sweep-out of $M'$. Our construction somewhat resembles
%a high dimensional analogue of the Birkhoff curve
%shortening process.

In Section~\ref{sec: proof of the width volume inequality}
we prove Theorem~\ref{thm: main} by inductively constructing
sweep-outs of larger and larger subsets of $M$.
The result of Section~\ref{sec: small submanifolds}
serves as the base of the induction.
%To do so, we carry out the necessary estimates to apply
%Theorem~\ref{thm: isoperimetric} and
%Proposition~\ref{prop: base of induction} to subdivide
%our manifold as needed.

In Section~\ref{sec: width of surfaces} we
prove Theorem \ref{BS}. We also describe how to obtain
a version of Theorem~\ref{thm: main} for manifolds,
which admit a conformal mapping into some nice space $M_0$.

In Section~\ref{sec:Volumes of hypersurfaces} we show
that a manifold with non-negative Ricci curvature can be
covered by balls of small $n$-volume, small
$(n-1)$-volume of the boundary, and such that the cover has
controlled multiplicity.
We use this decomposition to bound the volume of $k$-parametric
sweep-outs of $M$ and, consequently, volumes of
stationary integral varifolds or minimal hypersurfaces in $M$.

\begin{remark}
In \cite{Sabourau2014} Stephane Sabourau 
independently obtained upper bounds
on the width and volume of the smallest minimal hypersurface on Riemannian manifolds with
$\Ricci \geq 0$. 
\end{remark}

\subsection{Acknowledgements} We would like to thank Misha Gromov for suggesting
the idea of using the
methods of Korevaar's paper \cite{Korevaar1993}  in a similar context.
We would like to thank
Alexander Nabutovsky, Regina Rotman, and Robert Young for
valuable discussions and encouragement.

%%%%%%%%%%%%%%%%%%%%%%%%%%%%%%%%%%%%%%%%%%%%%%%%%%%%%
\section{Width of Riemannian manifolds}
\label{sec: definitions}
%%%%%%%%%%%%%%%%%%%%%%%%%%%%%%%%%%%%%%%%%%%%%%%%%%%%%

Let $G$ be an abelian group.
We denote the space of flat $G$-chains in $M$ by $F_k(M; G)$
and the space of flat $G$-cycles by $Z_k(M;G)$.
The space of integral flat chains was defined in \cite{FedererFleming1960}.
For flat chains with coefficients in an abelian group $G$ see \cite[(4.2.26)]{Federer1969}.
The deformation theorem of Federer and Fleming states that
a flat chain of finite mass and boundary mass
can be approximated by a piecewise linear
polyhedral chain (see \cite[(4.2.20),(4.2.20)$^{\nu}$]{Federer1969}).
The deformation theorem will be used
throughout this paper. Often we will abuse notation
and use the same letter for a flat $G$-chain and
a polyhedral chain approximating it.
We will denote the mass of a $k$-chain $c$
by $\Vol_k(c)$.

In \cite{Almgren1962} F. Almgren constructed an isomorphism
\[F_A: \pi_k(Z_{n-1}(M;G);0) \rightarrow H_{n+k}(M;G)\]
For $k=1$ the map $F$ can be described as follows.
Let $c_t \in Z_{n-1}(M;G)$, $t \in S^1$, be a continuous
family of cycles.
Pick a fine subdivision $t_0, ... ,t_m$ of $S^1$ and
let $C_i$ be a (nearly) volume minimizing
$n$-chain filling $c_{i}-c_{i-1}$ (for $i \in \Z_m$).
Then $C= \sum _{i \in \Z_m} C_i$ is an $n$-cycle.
It turns out that homology class of $C$ is independent
of the choice of the subdivision and filling chains $C_i$
as long as the subdivision is fine enough
and the mass of chains $C_i$ is close to the mass
of a minimal filling.

If $M$ is a manifold with boundary
we may also consider the space of flat cycles relative to the boundary of $M$.
Let $q$ be a quotient map $q: F_{k} (M; G) \rightarrow
 F_{k} (M,\partial M; G) = F_{k} (M; G)/F_{k} (\partial M; G) $.
The boundary map on $F_{k} (M; G)$ descends to a boundary
map $\partial$ on the quotient. This allows us to define the space of
relative cycles $Z_{k}(M, \partial M;G)$.
Cycles in this space can be represented by
$(n-1)$-chains with boundary in $\partial M$.
Almgren's map then defines an isomorphism
$\pi_1 (Z_{n-1}(M, \partial M;G),\{0\}) \cong H_{n-1}(M,\partial M;G)$.

For simplicity from now on we assume everywhere that
group $G = \Z_2$. Henceforth we drop the reference to the group $G$
from our notation. $\Z_2$ coefficients will suffice for all
applications to volumes of minimal surfaces
that we obtain in this paper.
When manifold $M$ is orientable the bound in Theorem~\ref{thm: main}
holds for sweep-outs with integer coefficients.
The proof is essentially the same
with some minor modifications to account for orientation
of cycles.

\begin{definition}
\label{defn: sweep-out}
We define the following two notions:

\begin{enumerate}
\item
For a closed manifold $M$ a map
$f: S^1 \rightarrow Z_{n-1}(M)$
is called a \emph{sweep-out} of $M$ if it is not contractible,
i.e. $F_A([f]) \neq 0$. Similarly, if $M$ has a boundary
we call $f: S^1 \rightarrow Z_{n-1}(M, \partial M)$
a sweep-out if the image of $[f]$
under Almgren's isomorphism is non-zero.

\item
The \emph{width of $M$} is
 \[\W(M) = \inf_{\{f\}} \sup_{t} \Vol_{n-1}(f(t))\]
where the infimum is taken over the set of all sweep-outs
of $M$.
\end{enumerate}

\end{definition}

For manifolds with boundary it will be convenient
to consider a particular
type of sweep-outs that start on a trivial cycle and
end on $\partial M$. We will call them $\partial$-sweep-outs.

\begin{definition} \label{partial width}
%We define the following of Definition~\ref{defn: sweep-out}:
\begin{enumerate}
\item Let $M$ be a manifold with boundary. A \emph{$\partial$-sweep-out
%(with coefficients in G)
of $M$} is a map $f: [0,1] \rightarrow Z_k(M)$, such that:
\begin{enumerate}
    \item $f(0)$ is a trivial $k$-cycle and $f(1)=\partial M$
    \item  Let $q \circ f: [0,1] \rightarrow Z_{n-1}(M, \partial M)$
    be the composition of $f$ with the quotient map $q$.
    When we identify $q \circ f(0)$ and $q \circ f(1)$
    we obtain a sweep-out of $M$.
\end{enumerate}
\item The \emph{$\partial$-width of $M$} is:
    \[\W^{\partial}(M) = \inf_{\{f\}} \sup_{t} \Vol_{n-1}(f(t)) \]
where the infimum is taken over the set of all $\partial$-sweep-outs
of $M$.
\end{enumerate}
\end{definition}

From the definition we have inequalities $\W(M) \leq\W^{\partial}(M)$
and $\W^{\partial}(M) \geq \Vol_{n-1} (\partial M)$.

Definition \ref{partial width} is motivated by the following
proposition.

\begin{proposition} \label{decomposition}
Let $U_0 \subset ... \subset U_{m-1} =M$ be a sequence
of nested open subsets of $M$, and let
$A_i$ denote the closure of $U_{i} \setminus U_{i-1}$ for $1 \leq i \leq m-1$
and $A_0$ denote the closure of  $U_0$.
Then
$\W^{\partial}(M) \leq \sup \{W^{\partial}(A_0), W^{\partial}(A_1)+\Vol_{n-1}(\partial U_0), ...,
W^{\partial}(A_{m-1})+\Vol_{n-1}(\partial U_{m-2})\}$.
\end{proposition}

\begin{proof}
By the definition of $\partial$-width
for each $i$
there exists a map $c_i:[0,1] \rightarrow Z_{n-1}(M)$
that starts on a trivial cycle, ends on $\partial A_i$
and is bounded in volume by $\W^{\partial}(A_i)+\epsilon$.
By definition of $A_i$,
$\partial A_i \subset \partial U_i \cup \partial U_{i-1}$
and $\partial U_i + c_{i+1}(1) = \partial U_{i+1}$.

We define a sweep-out $c:[0,1] \rightarrow Z_{n-1}(M)$ as follows.
For $0 \leq t \leq \frac{1}{m}$
we set $c(t) = c_0(t/m)$ and for
$\frac{i}{m} \leq t \leq \frac{i+1}{m}$, $i=1,...,m-1$
we set
$c(t)= c_i(m(t-\frac{i}{m}))+ \partial U_{i-1}$.

Let $F_A$ be the Almgren's isomorphism.
We can represent the homology class $F_A(c)$ by
a sum of $n$-chains $\sum_{i = 0}^{m-1} C_i$, such that
$\partial C_i = c(\frac{i}{m}) - c(\frac{i-1}{m}) $.
Moreover, since each $c_i$
is a $\partial$-sweep-out of $A_i$
we may assume that
$C_i$ represents a non-trivial homology class in
$H_n(M,\overline{M\setminus A_i}) \cong H_n(A_i,\partial A_i)$.

We claim that the sum $\sum_{i=0} ^k C_i$ represents
a non-trivial homology class in $H_n(M, M \setminus U_k)$.
Indeed, assume this to hold for $\sum_{i=0} ^{k-1} C_i$.
Let $V_1$ denote a small tubular neighbourhood
of the set $U_{k-1}$ inside $U_k$.
Let $V_2$ be a small tubular neighbourhood
of $U_k \setminus U_{k-1}$ inside $U_k$. Let $V_3 = V_1 \cap V_2 \subset U_k$.
The pair $(V_3, V_3 \cap \partial U_k)$ is homotopy equivalent to
$(\partial U_{k-1}, \partial U_{k-1} \cap \partial U_k)$.
From the Mayer-Vietoris sequence we have an isomorphism
$H_n(\overline{U_k}, \partial U_k ) \rightarrow ^{\partial}
H_{n-1}(\partial U_{k-1}, \partial U_{k-1} \cap \partial U_k)$.
This map sends $[\sum_{i=0}^{k-1} C_i + C_k]$ to the fundamental class
$[\overline{\partial U_{k-1} \setminus \partial U_k}]$.

We conclude that $c(t)$ is a $\partial$-sweep-out of $M$.
\end{proof}

In the last section of this paper we will obtain upper
bounds on the $k$-parametric sweep-outs $\W^k(M)$ of $M$.
%Following \cite{Guth2009} we define them as follows.
By Almgren's isomorphism theorem
we have $\pi_m(Z_{n-1}(M;Z);0)=0$ for $m>1$ and
$\pi_1(Z_{n-1}(M;\Z_2);0) \cong \Z_2$.
Hence, the connected component $Z_{n-1}^0$ of $Z_{n-1}(M;\Z_2)$
that contains the $0$-cycle is weakly homotopy equivalent to the
Eilenberg-MacLane space $K(\Z_2,1) \simeq \RP^{\infty}$.

Let $K$ be a $k$-dimensional polyhderal complex and
$\sigma: K \rightarrow Z_{n-1}^0(M)$ be continuous
and assume that $\sigma(x)$ has finite mass for all $x$. Following \cite{MarquesNeves2013} we 
define $k$-parametric width $\W^k$ as follows.

\begin{definition}
We introduce the following parametric version of Definition~\ref{defn: sweep-out}:
\begin{enumerate}
\item For a closed manifold $M$ we say that $\sigma$
is a $k$-parametric sweep-out of $M$ if
$\sigma(K)$ represents the non-zero class in $H_k(Z^0_{n-1},\Z_2) \cong \Z_2$.
\item Define the $k$-parametric width to be
$\W^{k}(M) = \inf_{\sigma} \sup_{t \in K} \Vol_{n-1}(f(t))$,
where the infimum is taken over the set of all $k$-parametric sweep-outs
$\sigma$.
\end{enumerate}
\end{definition}

It follows from the definition that $\W^1(M)=\W(M)$
and $\W^k(M) \leq \W^{k+1}(M)$.

Using Almgren-Pitts min-max theory it is possible to
obtain minimal hypersurfaces from sweep-outs of
$M$. In \cite{MarquesNeves2013} Marques and Neves
proved the following results.

\begin{theorem}[Marques-Neves] \label{Marques-Neves}
Let $M$ be a closed Riemannian manifold
of dimension $n$, $3 \leq n \leq 7$.

\begin{enumerate}
\item There exists a smooth, closed, embedded minimal
hypersurface in $M$ of volume $\leq \W(M)$.

\item If $\W^k(M)=\W^{k+1} (M)$ then there
exists infinitely many smooth, closed, embedded minimal
hypersurfaces in $M$ of volume $\leq\W^k(M)$.

\item Suppose $M$ is a manifold of positive Ricci curvature
and there exists only finitely many minimal hypersurfaces
of volume $\leq\W^k(M)$. Then there exists a smooth, closed, embedded minimal
hypersurface $\Sigma_k$ and $a_k \in \mathbb{N}$, such that
$a_k\Vol_{n-1}(\Sigma_k) =\W^k(M)$.

\end{enumerate}
\end{theorem}

\begin{remark}
In the proof of these results Marques and Neves
impose an additional
technical condition on $\W^k$.
Namely, they require that the infimum in the definition of
$\W^k$ is taken over only those maps $f:K \rightarrow Z_{n-1}(M)$
that have no concentration of mass.
This is defined as follows.
Using the notation of
\cite{Federer1969} let $\| c \|$ denote the Radon measure
associated with the flat chain $c$.
Then a map $f$ is said to have no concentration of mass if
$$\lim_{r \rightarrow 0}  \sup \{\|f(x)\|(B_r(a)): \ x\in K, a \in M \}=0$$

All estimates on $\W^k$ in our paper come from explicit
constructions of families of flat cycles
(in fact, polyhedral cycles), which have
no concentration of mass. Therefore
we can safely combine our estimates with the conclusions
of Theorem \ref{Marques-Neves}.
\end{remark}

%%%%%%%%%%%%%%%%%%%%%%%%%%%%%%%%%%%%%%%%%%%%%%%%%%%%%
\section{Isoperimetric inequality}
\label{sec:Isoperimetric inequality}
%%%%%%%%%%%%%%%%%%%%%%%%%%%%%%%%%%%%%%%%%%%%%%%%%%%%%

Let $(M,g_0)$ be a closed Riemannian $n$-manifold
with $\Ricci \geq -(n-1)$.
Let $g = \phi^2 g_0$ be a Riemannian
metric on $(M,g_0)$ in the conformal class of $g_0$.
Here $\phi : M_0 \rightarrow \mathbb{R}_+$ is a
    smooth function on $(M,g_0)$.  
    
\begin{notation}
We write $M_0$ for $(M,g_0)$ and $M$ for $(M, g)$.
\end{notation}

%The constant $C(M_0) = C(n) \max \{1,a \Vol(M_0)^{\frac{1}{n}}\} $
%in the statement of Theorem \ref{thm: main} is invariant 
%under scaling of the metric on $M_0$. 
%Hence, without any loss of generality we may assume
%that the Ricci curvature of $M_0$ satisfies $\Ricci \geq -(n-1)$.
%From now on we make this assumption everywhere in the paper.

Below we use the convention that geometric structures
measured with respect to $g_0$ have a superscript zero in
their notation. Geometric structures measured with respect
to $g$ have no superscript.

\begin{notation}
     Let $\Vol_k(U)$, $d(x,y)$, $dV$, $B(x,r)$ and
     $\nabla$ denote the $k$-volume function, distance
     function, volume element, closed metric ball of
     radius $r$ about $x$, and gradient with respect
     to $g$. Let $\Vol^0_k(U)$, $d^0(x,y)$, $dV^0$,
     $B^0(x,r)$, $\nabla^0$ denote the corresponding
     quantities with respect to $g_0$.
     
     % $B(x,r) = \{y \in M : d(y,x) \leq r\}$ denote a metric ball of radius $r$ about $x$, where distance is measured with respect to $g$, and $B^0(x,r)$ for a metric ball with respect to $g_0$.
\end{notation}

Let $W$ be a subset of $U$ and let $N^0_l(W)$
denote the set $\{x \in U| d^0(W,x)\leq l \}$.
%We will need the following lemma due to B. Colbois and D. Maerten \cite{ColboisMaerten}
%about finding domains with small neighbourhoods.

\begin{lemma} \label{subdivision}
There exists a set $W \subset U$ and $l \in (0,\frac{1}{2}]$,
such that
\begin{enumerate}
\item
$\Vol_n(U)/25^n\leq \Vol_n(W) \leq 2 \Vol_n(U)/25^n $

\item
$\Vol_n(N^0_l(W)) \leq  (1- \frac{1}{25^n}) \Vol_n(U)$

\item
$\Vol^0_n(N^0_l(W) \setminus W)
\leq l^n \max\{2 \Vol^0_n(U), c(n)\}$
\end{enumerate}
\end{lemma}

\begin{proof}
The argument is essentially the same as the proof of Lemma 2.2 in
the work of Colbois and Maerten \cite{ColboisMaerten}.
Let $r$ be the smallest radius with the property that $\Vol(B^0(a,r)\cap U)= \frac{\Vol(U)}{25^n}$ for some $a \in M$. 
%We have that for every $a' \in M$ the bound $\Vol_n(B^0(a',r) \cap U)\leq \frac{\Vol_n(U)}{25^n}$ holds.

We consider two cases. If $r\leq 1$ we define $W=B^0(a,r)\cap U$ and $l=\frac{r}{2}$.

We observe, using curvature comparison for the space $M_0$, that 
the $l$-neighbourhood of $B^0(a,r)$ can be covered by at most
$24.4^n$ balls of radius $r$.
Indeed, let $\{B^0(x_i,r/2) \}_{i=1}^N$ be a maximal collection of disjoint balls
with centers in $B^0(a,\frac{3r}{2})$. Since the collection
is maximal, the union $\bigcup B(x_i,r)$ covers $B^0(a,\frac{3r}{2})$.
Using the Bishop-Gromov comparison theorem we can estimate
the number $N$. Let $\Vol_n^0(B(x_j,\frac{r}{2})) = \min_i \{\Vol_n^0(B(x_i,\frac{r}{2})) \}$.
$$N \leq \frac{\Vol_n^0(B^0(a,\frac{3r}{2}))}{ \Vol_n^0(B(x_j,\frac{r}{2})) } \leq \frac{\Vol_n^0(B(x_j,\frac{5r}{2})) }{\Vol_n^0(B(x_j,\frac{r}{2})) } \leq \frac{V(\frac{5r}{2})}{V(\frac{r}{2})}$$
where $V(r)$ denotes the volume of a ball of radius $r$ in $n$-dimensional hyperbolic space.
When $r \in (0,1]$ this quantity is maximized for $r = 1$. We conclude that $B^0(a,\frac{3r}{2})$
can be covered by
$$N \leq \frac{\int_0 ^\frac{5}{2} \sinh^{n-1}(s)ds}{\int_0 ^\frac{1}{2} \sinh^{n-1}(s)ds}
\leq (2e^{\frac{5}{2}})^n \leq 24.4^n$$
balls, such that each of them has $\Vol_n(B^0(x_i,r)\cap U)$ at most $\frac{\Vol_n(U)}{25^n}$.
This proves inequalities (1) and (2) for the case $r\leq 1$.

Volume of a unit ball in hyperbolic 
$n$-space satisfies $V(1) \leq \omega_n e^{n-1}$,
where $\omega_n$ denotes the volume of a unit $n$-ball in Euclidean space.
Hence, $\Vol_n^0(B^0(a,\frac{3r}{2})\setminus B^0(a,r)) \leq 25^n e^{n-1} \omega_n r^n = c(n)$. This proves (3) for the case $r\leq 1$.

Suppose $r> 1$. Let $k$ be the smallest number, such that there
exists a collection of $k$ balls of radius $1$
$\{B^0(x_i,1) \}_{i=1}^k$ with $\Vol(\bigcup B^0(x_i,1)\cap U) \geq
\frac{\Vol_n(U)}{25^n}$. Let $\{B^0(x_i,1) \}_{i=1}^k$ be a collection
of $k$ balls with the property that if $\{B^0(y_i,1) \}_{i=1}^k$
is any other collection of $k$ balls then
$\Vol(\bigcup B^0(x_i,1)\cap U)\geq \Vol(\bigcup B^0(y_i,1)\cap U)$.
We set $W = \bigcup B^0(x_i,1)\cap U$. Note that by our definition
of $k$ we have $\Vol_n(W) < \frac{2\Vol_n(U)}{25^n}$.

Consider $1/2$-neighbourhood of $W$ and note that it can be
covered by at most $(24.4)^n$ sets $B_j$, where
each $B_j$ is a union of $k$ balls $B^0(y_i,1)$ of radius $1$.
By definition of $W$ we have $\Vol(B_j)\leq \Vol(W)$,
so $\Vol_n(N^0_l(W)) \leq \frac{24.4^n+1}{25^n} \Vol(U)$.
Finally, we observe that $\frac{\Vol^0(N^0_l(W))}{1/2}\leq 2 \Vol^0(U)$.
\end{proof}

\begin{theorem} \label{thm: isoperimetric}
    There exists a constant $c(n)$ such that the
    following holds: Let $U \subseteq M$ be an open
    subset.  There exists an $(n-1)$-submanifold $\Sigma \subset U$ subdividing $U$ into two open sets $U_1$ and $U_2$ such that $\Vol_n(U_i) \geq (\frac{1}{25^n}) \Vol_n(U)$ and $\Vol_{n-1}(\Sigma) \leq c(n)
\max\{1,\Vol^0_n(U)^{\frac{1}{n}}\} \Vol_n(U)^{\frac{n-1}{n}}$.
\end{theorem}

\begin{proof}
We use the length-area method (see \cite[p. 4]{Gromov1983})
to find a small volume hypersurface in $N^0_{l}(W) \setminus W$,
where $W$ and $l$ are as in Lemma \ref{subdivision}.

%In particular, this bound holds for $a' \in U$.
Let $f(x)= d^0(W,x) |_{U} : U \rightarrow \mathbb{R}^+$
be the $d^0$ distance form $x$ to $W$ restricted to the set $U$.
By Rademacher's theorem, $f$ is differentiable almost everywhere.
By applying the co-area formula we have:

\begin{eqnarray*}
    \int_0 ^{l} \Vol_{n-1}(f^{-1}(t)) dt & = &  \int_{f^{-1}(0,l)} ||\nabla f|| dV\\
    (\text{H\"older's inequality})    & \leq & \left(\int_{f^{-1}(0,{l})} ||\nabla f||^n dV\right)^{\frac{1}{n}} \left( \Vol_n(f^{-1}(0,{l}) \right)^\frac{n-1}{n}\\
        & = & \left( \Vol^0 _n (f^{-1}(0,{l}))\right)^{\frac{1}{n}} \left(\Vol_n(f^{-1}(0,{l}))\right)^{\frac{n-1}{n}}
\end{eqnarray*}

The last equality holds since $||\nabla f||^n dV = ||\nabla^0 f||^n dV^0$ is a conformal invariant.
By Lemma \ref{subdivision} we have
$\Vol^0_n(f^{-1}(0,l))^{\frac{1}{n}} 
\leq c(n) l \max\{\Vol^0_n(U)^{\frac{1}{n}}, 1\}$.
For the second factor we apply the bound $\Vol_n(f^{-1}(0,{l}))\leq \Vol_n(U)$.
It follows that
$$\min _{r<t<2r} \Vol_{n-1}(f^{-1}(t)) \leq
c(n) \max\{\Vol(U)^{\frac{1}{n}}, 1\} \Vol_n(U)^{\frac{n-1}{n}}$$
Thus for some regular value of $t$ the level set $f^{-1}(t)$ with area no larger than average, is the desired submanifold $\Sigma$. We take $U_1 = f^{-1}([0,t))$ and $U_2 = f^{-1}((t,\infty))$.

Since $W \subseteq U_1$ by Lemma \ref{subdivision} we have
$\Vol(U_1) \geq \frac{\Vol_n(U)}{25^n}$. On the other hand,
$U_1 \subseteq N^0_{l}(W)$ of volume at most $1-\frac{\Vol_n(U)}{25^n}$
so $\Vol(U_2) \geq \frac{\Vol_n(U)}{25^n}$.
\end{proof}

%%%%%%%%%%%%%%%%%%%%%%%%%%%%%%%%%%%%%%%%%%%%%%%%%%%%%
\section{The width of small submanifolds}
\label{sec: small submanifolds}
%%%%%%%%%%%%%%%%%%%%%%%%%%%%%%%%%%%%%%%%%%%%%%%%%%%%%

In this section we will show that if a submanifold $M'$ of a Riemannian
manifold $M$ has small enough volume then its $\partial$-width can be bounded from above
in terms of $\Vol_{n-1}(\partial M')$.  First we show this for a submanifold
that is contained in a very small ball.

\begin{definition}
    For a closed Riemannian manifold $M$ and $\epsilon_0 \in (0,1)$
    define $\epsilon(M, \epsilon_0)$ to be the largest radius
    $r$ such that: for every $x \in M$ we have that $B(x,r)$
    is $(1+\epsilon_0)$-bilipschitz diffeomorphic to the
    Euclidean ball of radius $r$.
\end{definition}

%\begin{notation}
%    We write $| z |$ for the mass of $z$ of a flat integral
%    Lipschitz chain $z$.
%\end{notation}

\begin{lemma} \label{euclidean}
If $M'\subset M$ is contained in a ball of radius
$\epsilon(M, \epsilon_0)$ then
$\W^{\partial}(M') \leq (1+\epsilon_0) \Vol_{n-1}(\partial M')$.
\end{lemma}

\begin{proof}
A $2$-dimensional version of the lemma appeared in \cite{Liokumovich2013}.  Let
$U\subset \R^n$ be the image of $M'$ under $(1+\epsilon_0)$-bilipschitz
diffeomorphism $F$.  An argument similar to that in \cite[\S 6]{Milnor1963}
shows that for a generic line $l \in \R^n$ the projection of $\partial U$ onto
$l$ is a Morse function.  Let $p$ denote such a projection map and assume that
$p(U)=[0,c]$.

Define $f: [0,c] \rightarrow Z_{n-1}(U,\Z)$ by setting
$$f(t) = \partial ( p^{-1}([0,t]) \cap U)$$

Open subsets of hyperplanes in $\R^n$ are volume minimizing
regions. Therefore we have $\Vol_{n-1}(f(t)) \leq \Vol(U)$
for all $t$. Composing $f$ with $F^{-1}$
we obtain the desired sweep-out.

\end{proof}

We extend the result of the lemma to submanifolds
of small volume.

\begin{proposition} \label{prop: base of induction}
     There exist a constant $C_1(n) > 0$, such that
     for every closed Riemannian $n$-manifold $M$,
     $\epsilon_0 >0$ and every embedded submanifold
     $M' \subset M$ of dimension $n$ and volume
     $\Vol_{n}(M') \leq \epsilon(M,\epsilon_0)^n / C_1$
     the following bound holds:
    \[ \W ^{\partial}(M') \leq
     3 (1+\epsilon_0) \Vol_{n-1}(\partial M') \]
\end{proposition}

The proof of Proposition \ref{prop: base of induction}
somewhat resembles a high dimensional analog of
the Birkhoff curve shortening process.
We cover $M'$ by a finite collection of small balls $B_i$
such that balls of $1/4$ of the radius still cover $M'$.
Since $M'$ has very small volume it will not contain any
of the balls $B_i$. Hence, we can cut away the part of $\partial M'$
that is contained in $B_i$ and replace it with a minimal
surface that does not intersect $(1/4)B_i$. As a result we obtain a new
submanifold $M'' \subset M'$ that does not intersect $(1/4)B_i$.
Moreover, we can do this in such a way that
volume of the boundary does not increase.
The difference $M' \setminus M''$
is contained in a small ball, so we can sweep
it out by Lemma \ref{euclidean}.
After finitely many iterations we obtain
a submanifold that is entirely contained in one of the small balls.
We then apply Proposition \ref{decomposition}
to assemble a sweep-out of $M'$ from sweep-outs
in small balls.

In the proof of Proposition~\ref{prop: base of induction}
we will need the following isoperimetric inequality:

\begin{theorem}[Federer--Fleming]
    There exists a constant $C_2(n)>1$, such that
    every $k$-cycle $A$ in a closed unit ball in
    $B\subset \mathbb{R}^n$
    can be filled by a $(k+1)$-chain $D$ in $B$, such that:
    (i) $\Vol(D)\leq C_2(n)\Vol(A)^{\frac{k+1}{k}}$, and
    (ii) $D$ is contained in the $C_2(n)\Vol(A)^{\frac{1}{k}}$-neighbourhood  of $A$.
\end{theorem}

%Below, we will only need the cases $n-2 \leq k \leq n$.
%Thus we may take $C_3(n) \leq \frac{ 1 }{ 2 } n^3$.

To show Proposition~\ref{prop: base of induction} we first need to
prove the following lemma.

\begin{definition}
A $k$-chain $A$ will be called $\delta$-minimizing
if $\Vol(A) - \delta \leq \inf \{A' \in C_k(M,\Z): \partial A' = \partial M \}$.
\end{definition}

\begin{lemma} \label{lemma: filling in the annulus}
    There is a constant $C_3(n)$ such that the following holds:
    Let $B$ be a ball of radius $r_0 \leq \epsilon(\epsilon_0,M)$
    and  $A \subset \partial B$ be an $(n-1)$-chain satisfying
    $\Vol(A) \leq C_3(n) \Vol(\partial B)$.
    For every $\delta > 0$ there exists $\delta$-minimal filling
    $D$ of $\partial A$ in $B(x,r_0)$, such that
    $D \cap B(x,r_0/2) = \emptyset$.
    We may take $C_3(n) \leq \omega^{-1}_{n-1} (10 C_2(n))^{-n}$
\end{lemma}

The proof of Lemma~\ref{lemma: filling in the annulus} is a variation
of an argument in \cite[\S 4.2-3]{Gromov1983}. See also \cite[Lemma 6]{Guth2006}.

\begin{proof}
Fix $\delta' < \delta r_0 / 100C_2(n)$. Let $D_1$ be some $\delta'$-minimal
filling of $\partial A$ in $B$. We claim that $D_1$ is contained in a
$r_0/4$-neighbourhood of $\partial A$ except for a subset of volume at
most $\delta'$.

Since $B$ is 2-bilipschitz homeomorphic to a Euclidean ball, we may
apply the Federer-Fleming isoperimetric inequality (with a worse
constant) inside $B$. We obtain that every $(n-2)$-cycle $S$ can be
filled in $B$ by an $(n-1)$-chain of mass at most
$2 C_2(n)\Vol(S)^{\frac{n-1}{n-2}}$.

Let $A(r) = \Vol_{n-2}(\{x \in D_1 : d(x, \partial A) = r \})$ and
$V(r) = \Vol_{n-1}(\{x \in D_1 : d(x, \partial A) > r \})$.
The co-area inequality implies that $|V'(r)| \geq A(r)$.

It follows by the $\delta'$-minimality of $D_1$ that every open
subset $U \subset D_1$ not meeting $\partial A$ must have
volume at most:
\[\Vol_{n-1}(U) \leq  2 C_2(n)\Vol_{n-2}(\partial U)^{\frac{n-1}{n-2}} + \delta' \]

In particular, we have:
$V(r) \leq 2C_2(n) A(r)^{\frac{n-1}{n-2}} + \delta'$.
Applying the co-area inequality again we obtain:

\[ \frac{d}{dr} \left([V(r) - \delta']^{\frac{1}{n-1}}\right)
\leq \frac{-1}{(n-1) (2C_2(n))^{\frac{n-1}{n-2}}}\]

Hence, $V(r) \leq \delta'$ for some
\begin{eqnarray*}
    r & \leq & (n-1) (2C_2(n))^{\frac{n-1}{n-2}}\Vol(D_1)^{\frac{1}{n-1}}\\
    & \leq & (n-1) (2C_2(n))^{\frac{n-1}{n-2}}\Vol(A)^{\frac{1}{n-1}}\\
    & \leq &  2 (n-1) (2C_2(n))^{\frac{n-1}{n-2}} \left( C_3(n) n \omega_{n} r_0^{n-1}\right)^{\frac{1}{n-1}} \leq r_0/4\\
\end{eqnarray*}

We will now cut off the piece of $D_1$ that lies outside of
$(r_0/4)$-neighbourhood of $\partial D_1$.  Again, by the co-area
inequality we have that: $A(r') \leq \frac{8}{r_0} \delta'$
for some $(1/4) r_0 \leq r' \leq (3/8) r_0$.
The Federer-Fleming isoperimetric inequality gives a filling of
$\{d(x, \partial D_1) = r' \}$ by
an $(n-1)$-chain $D_2$ satisfying:
\[\Vol(D_2) \leq
    2 C_2(n) \left(\frac{8}{r_0} \delta'\right)^{\frac{n-1}{n-2}}
    \leq \delta/2 \]
Moreover, the filling has the property that the distance from
$\{d(x, \partial D_1) = r' \}$ to every point of
$D_2$ is at most $2 C_2(n) \left(\frac{8}{r_0} \delta'\right)^{\frac{1}{n-1}}\leq r_0/8$.
This gives the desired filling.
\end{proof}

Now we prove Proposition \ref{prop: base of induction}.
We will construct a decomposition of $M$ into open sets
and then apply Proposition \ref{decomposition}.

\begin{proof}
Set $C_1(n) = 4^n \omega_{n-1} (10 C_3(n))^n$.
Let $\epsilon = \epsilon(M,\epsilon_0)$ and assume that $M' \subset M$
has volume bounded by $1/C_1(n) \epsilon^n$. Let $B_i=B(x_i,\epsilon)$ for $i=1,...,N$, be a collection of balls such that $M'$ is contained in the interior of  $\bigcup B(x_i,\epsilon/4)$. Fix $\delta>0$. We will construct a collection of open subsets $U_1 \subset ... \subset U_{N}$, with the following properties:

\begin{enumerate}
    \item $U_{N} = M'$.
    \item $\Vol(\partial U_{i})\leq \Vol(\partial U_{i+1}) + \delta / 2^i$.
    \item $U_i \cap \bigcup_{j=i+1}^{N} B(x_j,\epsilon/4)$ is empty.
\end{enumerate}

Assume that $U_{i+1},...,U_N$ have been defined.
If $U_{i+1} \cap B(x_i,\epsilon/4)$ is empty
we set $U_i = U_{i+1}$.
 Otherwise, to construct $U_i$ we proceed as follows.

%Otherwise, consider $(n-2)$-cycle
%$C=\partial Z_{i-1} \cap \partial B(x_i, \epsilon)$.

\begin{figure}
   \centering  
    \includegraphics[scale=1.2]{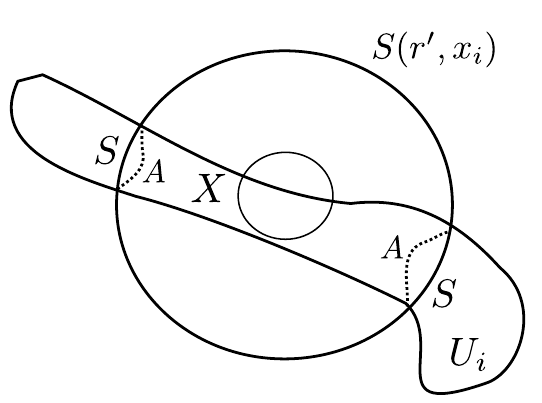}
    \caption{Proof of Proposition \ref{prop: base of induction}}
\end{figure}

By the co-area inequality we can find $S(x_i,r') =\partial B(x_i,r')$, with $\frac{3}{4} \epsilon < r'< \epsilon$, such that
$S=U_{i+1} \cap S(r',x_i)$ satisfies
$\Vol_{n-1}(S) \leq 4 \Vol_n(U_{i+1} \cap B(x_i, \epsilon))^{1-1/n}$.
By Lemma~\ref{lemma: filling in the annulus}
there exists an $(n-1)$-chain $A \subset B(x_i,r')$
with $\partial A = \partial S$ which is $(\delta/2^i)$-minimizing and $A$ does not intersect $B(x_i, \epsilon/4)$.
Let $X$ denote the union of the connected components
of $U_{i+1} \setminus A$ that intersect $B(x_i, \epsilon/4)$.
We define $U_i = U_{i+1} \setminus X$.
Note that the volume decreased and
by $\delta/2^i$-minimality of $A$ and the volume of the
boundary could not have increased by more than $\delta/2^i$.

By Lemma \ref{euclidean} we have
$\W^{\partial}(X) \leq 2 (1+\epsilon_0) \Vol_{n-1}(\partial M')+ \delta$.
By Proposition \ref{decomposition} we have
$\W^{\partial}(M) \leq 3 (1+\epsilon_0) \Vol_{n-1}(\partial M')+ 2\delta$.
Since $\delta$ can be chosen arbitrarily small this concludes
the proof of Proposition \ref{prop: base of induction}.

\end{proof}

%\begin{lemma} \label{isoperimetry for small subsets}
%Every submanifold $M \subset M_0$ of volume
%$|M|\leq \epsilon(M_0)^n$ satisfies
%$|\partial M| \geq 1/10 \epsilon(M_0)^{n-1}$.
%\end{lemma}
%
%\begin{proof}
%
%\end{proof}

\section{Proof of the width inequality}
\label{sec: proof of the width volume inequality}

In this section we prove Theorem \ref{thm: main}.

\begin{theorem} \label{thm: main with boundary}
Let $M_0$ be a manifold with $\Ricci \geq -(n-1)$ and let
$M$ be in the conformal class of $M_0$. Let
$M' \subseteq M$ be an $n$-dimensional submanifold.
There exists a constant $C(n)$ that depends on the dimension,
such that:
\[ \W^{\partial}(M') \leq
C(n) \max \{1,\Vol_n ^0(M')^{\frac{1}{n}} \} \Vol_n(M')^{\frac{n-1}{n}}+3 \Vol_{n-1}(\partial M') \]
\end{theorem}

Theorem~\ref{thm: main} follows as a special case.

\begin{proof}
Pick the constant $C(n)=4 \cdot 25^n c(n)$,
where $c(n)$ is the constant in Theorem \ref{thm: isoperimetric}.

Let $\epsilon > 0$ be small enough that every submanifold of volume at most
$25^n \epsilon$ satisfies conclusions of Theorem \ref{prop: base of induction}.
Suppose that $M' \subseteq M$, and pick $k$ so that: $k \epsilon < \Vol(M') \leq (k+1)
\epsilon$ and $k>25^n$. We proceed by induction on $k$.

Assume the desired sweep-out exists for every open subset of volume at most $k \epsilon$.
By Lemma \ref{thm: isoperimetric} we can find an $(n-1)$-submanifold $\Sigma$ subdividing $M'$ into $M_1$ and $M_2$ of volume at most $c(n) \max \{1, \Vol^0_n(M')^{\frac{1}{n}} \} \Vol_n(M')^{\frac{n-1}{n}}$, such that $\Vol_n(M_i)\leq (1- 1/25^n) \Vol_n(M')$.
Since $k > 25^n$ the inductive hypothesis is applicable to both halves $M_i$.

By inductive hypothesis we have
\[ W^{\partial}(M_i) \leq 3(\Vol(\partial M' \cap M_i) + \Vol(\Sigma)) +
C(n) \max \{1,\Vol_n(M)^{\frac{1}{n}} \} \Vol_n(U_i)^{\frac{n-1}{n}}
\]

We apply Proposition \ref{decomposition} with $U_0 = M' \setminus M_2$
and $U_1 = M'$. We obtain
\[\W^{\partial}(M') \leq 3 \Vol(\partial M')+ 4 \Vol(\Sigma) + 
C(n) \max\{1, \Vol_n^0(M')^{\frac{1}{n}} \} \max_{i=1,2} \{\Vol^n(M_i)^{\frac{n-1}{n}}\}\]

We use bounds $\Vol_n(M_i)^{\frac{n-1}{n}} \leq \frac{25^n-1}{25^n} \Vol_n(M')$ 
and 
$$\Vol_{n-1}(\Sigma) \leq c(n) \max \{1, \Vol^0_n(M')^{\frac{1}{n}} \} Vol_n(M')^{\frac{n-1}{n}}$$

We compute that the resulting expression satisfies the desired bound.

%After substituting $C(n,\nu) \Vol(M')^{\frac{n-1}{n}}$ for
%$\Vol(\Sigma)$
%and $\left( 1-\frac{1}{\tau +2} \right)\Vol(M')^{\frac{n-1}{n}}$ for
%$\max_{i=1,2} \{\Vol(M_i)^{\frac{n-1}{n}}\}$
%we obtain the desired bound.
\end{proof}

Theorem \ref{thm: main} follows from 
Theorem \ref{thm: main with boundary} by taking the infimum
of the total volume of $M_0$ over all 
manifolds $M_0$ that are conformally equivalent 
to $M$ and
have $\Ricci \geq -(n-1)$.

%%%%%%%%%%%%%%%%%%%%%%%%%%%%%%%%%%%%%%%%%%%%%%%%%%%%%
\section{The width of surfaces} % (fold)
\label{sec: width of surfaces}
%%%%%%%%%%%%%%%%%%%%%%%%%%%%%%%%%%%%%%%%%%%%%%%%%%%%%

In this section we prove a theorem of Balacheff and Sabourau
\cite{BalacheffSabourau2010} with an improved constant.
Note that the 
result also follows as an immediate corollary of Theorem \ref{thm: main}
with a worse constant.
However, we observed that one can use a slightly modified version
of our proof and invoke the Riemann-Roch theorem
to get a somewhat better constant.

Below we prove a version of Theorem~\ref{thm: main} which
allows us to bound width of a manifold $M$
if $M$ admits a conformal map into some nice space $M_0$
with a small number of pre-images. We will then
estimate the width of surfaces by applying uniformization theorem and the Riemann-Roch theorem. 
%At the end of this section, we will recover F. Balacheff
%and S. Sabourau's result, Theorem~\ref{BS}, with an improved constant. 
Our argument
is parallel to the analogous arguments of Yang and Yau \cite{YangYau} and \cite[\S 4]{Korevaar1993} for
eigenvalues of the Laplacian on Riemann surfaces.

\begin{definition}\label{definition: covering constants}
    Define $\tau =\tau(M_0)$ and $\nu=\nu(M_0)$ as
    follows: $\tau$ is the least number such that
    any annulus $B^0(x,2r) \setminus B^0(x,r)$ in $M_0$
    can be covered by $\tau$ balls of radius $r$.
    We let $\nu(M_0)$ be the least constant such that
    $\Vol_n^0 (B^0(x,r)) \leq \nu r^n$ for all $r > 0$
    and all $x \in M_0$.
\end{definition}

\begin{theorem} \label{thm: mappings}
Let $\Phi : (M,g) \rightarrow (M_0,g_0)$ be a conformal map.
Suppose the following holds:
(i) Any point $x \in M_0$ has at most $K$ pre-images,
(ii) The set $\{x \in M, d\Phi(x)=0 \}$ is of measure 0. It follows that:
\[ \W(M)
    \leq \frac{8 \nu^{\frac{1}{n}} K^{\frac{1}{n}}}
    {1-(\frac{\tau+1}{\tau+2})^{\frac{n-1}{n}}}  \Vol(M)^{\frac{n-1}{n}}
\]
\end{theorem}

\begin{proof}
First, we prove an analog of our isoperimetric inequality,
Theorem~\ref{thm: isoperimetric}.

Let $U$ be an open set in $M$. We show that there is an $(n-1)$-submanifold
$\Sigma \subset U$ such that $U \setminus \Sigma = U_1 \sqcup U_2$ with
$\Vol_n(U_i) \geq \frac{ 1 }{ \tau + 2 } \Vol_n(U)$ and $\Vol_{n-1}(\Sigma)
\leq 2 \nu^{\frac{1}{n}} K^{\frac{ 1 }{ n } }\Vol_n(U)^{\frac{ n-1 }{ n }}$.

Let $p \in M$
and $u$ and $v$ be vectors in the tangent space $T_p M$.
Since $\Phi$ is conformal we have
 $$\langle \Phi_* u, \Phi_* v \rangle_{g_0} = \phi(x) \langle u, v \rangle_g$$
for some non-negative function $\phi$.
%Moreover, $\phi$ is positive on a set of full measure
%$M \setminus \{x \in M, d\Phi(x)=0 \}$.
In a neighbourhood of a point $p \in M \setminus \{x \in M, d\Phi(x)=0 \}$
map $\Phi$ is a local diffeomorpism and

\[ ||\nabla (f \circ \Phi)|| = \phi^{1/2} ||\nabla f||
    \quad dV_{g} = \phi^{-n/2} dV_{g_0} \]

where $f: M_0 \rightarrow \R$ is a smooth function and
$dV_{g}$, $dV_{g_0}$ are volume elements.

The fact that the measure of the
set $\{x \in M, d\Phi(x)=0 \}$ is zero guarantees that
$\lim_{r \rightarrow 0} \Vol_n(\Phi^{-1}(B^0(a,r))) =0 $
for all $a \in M_0$.
Let $r$ be the smallest radius, such that there exists a ball
$B(r,a)$ with
$\Vol(\Phi^{-1}(B^0(a,r)) \cap U)= \Vol_n(U)/(\tau + 2)$.

Let $d^0$ be the distance function on $M_0$ and define
$f(x) = d^0(a, x) |_{\Phi(U)} : \Phi(U) \rightarrow \mathbb{R}^+$
to be the distance from $x\in \Phi(U) \subset M$ to $a$.

\begin{eqnarray*}
    \int_r ^{2r} \Vol_{n-1}( (f \circ \Phi)^{-1} (t)) dt & = &  \int_{( f \circ \Phi)^{-1}(r,2r)} ||\nabla (f \circ \Phi)|| dV_{g}\\
    & \leq & \left(\int_{(f \circ \Phi)^{-1}(r,2r)} ||\nabla (f \circ \Phi)||^n dV_g\right)^{\frac{1}{n}} \left( \Vol_n((f \circ \Phi)^{-1}(r,2r) )\right)^\frac{n-1}{n}\\
    & \leq &  K ^{\frac{1}{n}} \left(\int_{f^{-1}(r,2r)}  ||\nabla^0 f||^n  dV_{g_0} \right)^{\frac{1}{n}} \left( \Vol_n((f \circ \Phi)^{-1}(r,2r) )\right)^\frac{n-1}{n}\\
%    & = & \left(\int_{(f \circ \Phi)^{-1}(r,2r)} ||df||^n dV^0 \right)^{\frac{1}{n}} \left( \Vol_n((f \circ \Phi)^{-1}(r,2r) )\right)^\frac{n-1}{n}\\
%    & = & {\color{red} \text{magic} }\\
    & \leq & 2 r \nu^{\frac{1}{n}} K^{\frac{1}{ n }} \Vol(U)^{\frac{ n-1 }{ n } }
\end{eqnarray*}

It follows that the average of $\Vol_{n-1}((f \circ \Phi)^{-1}(t))$ is smaller
than $2 \nu^{\frac{1}{n}} K^{\frac{1}{ n }} \Vol(U)^{\frac{ n-1 }{ n } }$. We then
take $\Sigma = (f \circ \Phi)^{-1}(t)$, with area at most average.
This finishes the proof of the analog of Theorem~\ref{thm: isoperimetric}.

The rest of the proof of Theorem~\ref{thm: mappings}
proceeds exactly as in Section~\ref{sec: proof of the width volume inequality}
with $c(n) \max\{1,\Vol^0_n(U)^{\frac{1}{n}}\}$ replaced by $2 r \nu^{\frac{1}{n}} K^{\frac{1}{ n }}$.
\end{proof}

We now recover Theorem~\ref{BS}.  Let $S_g$ denote a genus $g$ closed surface with a
complete Riemannian metric. We write $h$ for the metric on $S_g$.

The uniformization theorem for Riemannian surfaces guarantees that there is a metric $\phi h$ of constant sectional curvature 
in the conformal class of
$h$.  If $g = 0$ or $g=1$ then the result 
follows from Theorem \ref{thm: mappings}
by taking $M_0$ to be $S^2$, $\RP^2$, $T^2$ or
the Klein botle $K$ with the standard metric.
In all of these cases we have $\nu= \pi$ and $\tau = 6$
(see Remark~\ref{rem: refine BS} below).

Suppose now that the surface is orientable and the genuse $g>1$.
Take $\phi h$ to
have constant sectional curvature $\kappa = -1$.  We now apply the Riemann-Roch
theorem which gives a meromorphic function $\Phi: S_g \rightarrow S^2$ of
degree at most $g+1$. Since $\Phi$ is a ramified conformal covering map,
it has at most $g+1$ points in each fiber and there are finitely
many points where $d\Phi = 0$.
%We write $g_{S^2}$ for metric on the round sphere of radius one.
Applying Theorem~\ref{thm: mappings} to $\Phi$ gives a
width volume inequality for surfaces of genus $g > 1$, we obtain:
\[ \W(S_g) \leq \frac{8 \sqrt{\nu(S^2)}}
    {1-\sqrt{\frac{\tau(S^2)+1}{\tau(S^2)+2}}} \sqrt{(g+1)\Area(S_g)} \]

%If $S_g$ is a non-orientable surface of genus $g$ then we can apply the above argument to the oriented double cover $S'$ of $S_g$.
%One can check that the genus of $S'$ is $g-1$. 
%A sweepout of $S'$ will descend to a sweepout of $S_g$.
%Since the volume of $S'$ is twice the volume of $S_g$
%it follows that 
%\[ \W(S_g) \leq \frac{16 \sqrt{\nu(S^2)}}
%    {1-\sqrt{\frac{\tau(S^2)+1}{\tau(S^2)+2}}} \sqrt{g \Area(S_g)} \]

\begin{remark} \label{rem: refine BS}
Clearly $\nu(S^2) = \nu(\R^2) = \pi$.
It is well known that the smallest number of discs
of radius $1$ required to cover an annulus $B(2) \setminus B(1) \subset \R^2$
is $6$.
A similar covering also works on $S^2$ so $\tau(S^2)=\tau(\R^2)=6$.
With these values of $\tau$ and $\nu$ we compute 
$\frac{8 \sqrt{\nu(S^2)}}
    {1-\sqrt{\frac{\tau(S^2)+1}{\tau(S^2)+2}}}\leq 220$ which improves the upper bound $C \leq 10^8$ from \cite{BalacheffSabourau2010}.  \end{remark}

%%%%%%%%%%%%%%%%%%%%%%%%%%%%%%%%%%%%%%%%%%%%%%%%%%%%%
\section{Volumes of hypersurfaces}
\label{sec:Volumes of hypersurfaces}
%%%%%%%%%%%%%%%%%%%%%%%%%%%%%%%%%%%%%%%%%%%%%%%%%%%%%

In this section we prove Theorem \ref{hypersurfaces}.

\begin{theorem} \label{parametric width}
If $M$ is a manifold with non-negative Ricci curvature then 
$\W^k(M) \leq C(n) k^{\frac{1}{n}} \Vol(M)^{\frac{n-1}{n}}$.
\end{theorem}

Note that Theorem \ref{parametric width} is consistent with 
the conjecture that the sequence of numbers $\W^k(M)$
obeys a Weyl type asymptotic formula (see \cite{Gromov1988} and 
the discussion in
\cite[\S 9]{MarquesNeves2013}).

To prove Theorem \ref{parametric width} we will 
need to decompose $M$ into open subsets of small sizes.
Similar arguments for bounding $\W^k$ have been used
by Gromov~\cite{Gromov1988},\cite{Gromov2003}
and Guth~\cite{Guth2009}.

\begin{lemma}
Let $M$ be a closed Riemannian manifold with
$\Ricci(M) \geq 0 $.
There exists a constant $C_4(n)$, such that for any $p$
there exists $p' \leq p$ and a collection of open balls
$\{U_i\}_{i=1} ^{p'} $
with $\bigcup U_i = M$, $\Vol_n (U_i) \leq C_4(n) \frac{\Vol_n (M)}{p}$ and
$\Vol_{n-1}(\partial U_i) \leq C_4(n) \left( \frac{\Vol_n(M)}{p} \right)^{\frac{n-1}{n}}$.
\end{lemma}

\begin{proof}
It is a standard fact in comparison geometry
that for any ball $B(x,r) \subset M$ we have
$\Vol_n(B(x,3r)) \leq 3^n \Vol(B(x,r))$ and
$\Vol_{n-1}(\partial B(x,3r))\leq 3^{n-1} n \omega_n ^{\frac{1}{n}} 
\Vol_n(B(x,r))^{\frac{n-1}{n}}$.

Both of these bounds can be deduced, for example,
from the Bishop-Gromov inequality
\[\frac{\Vol_n (B(x,r-\epsilon))}{\omega_n (r-\epsilon)^n} 
\geq \frac{\Vol_n (B(x,r))}{\omega_n r^n} \]
where $\omega_n$ denotes the volume of a unit ball in
Euclidean $n$-space.

To prove the second bound observe that
$\Vol_{n} (B(x,r)  \setminus  B(r-\epsilon))
\leq \frac{n \epsilon}{r}\Vol_n (B(x,r)) + O(\epsilon^2)$.
Since $\Vol_n (B(x,r))\leq \omega_n r^n$ we can bound 
the volume of the annulus by 
$ n \omega_n ^{\frac{1}{n}} \epsilon 
(\Vol_n (B(r)))^{\frac{n-1}{n}} +  O(\epsilon^2)$.
Since $\Vol(3 B_i) \leq 3^n \Vol(B_i) $  we obtain that for every $\epsilon>0$
the volume of the annulus $B(x_i, 3r_i) \setminus B(x_i, 3r_i- \epsilon)$ is bounded by $3^{n-1} n \omega_n ^{\frac{1}{n}} \epsilon  \Vol(B(x_i, r_i))^{\frac{n-1}{n}}+O(\epsilon^2)$.
Hence, there must exist a sphere $S(x,r')$ in the annulus,
$3r-\epsilon \leq r' \leq 3r$,
with $\Vol_{n-1}(S(x,r')) \leq 3^{n-1} n \omega_n ^{\frac{1}{n}} \Vol_{n} (B(x,r))^{\frac{n-1}{n}}+O(\epsilon^2)$.
By curvature comparison again the volume of a sphere can not suddenly jump up.
Since $\epsilon$ was arbitrary we conclude
\[\Vol_{n-1}(\partial S(x,3r)) \leq3^{n-1} n \omega_n ^{\frac{1}{n}} 
\Vol_n(B(x,r))^{\frac{n-1}{n}}\]

Now we construct a covering of $M$ by disjoint balls of
volume $\frac{\Vol_n(M)}{p}$, such that balls of three times
the radius cover $M$.
This is also standard (see \cite{Gromov2006}).
For each $x$ choose $r_x>0$ to be the radius
of a ball $B(x,r_x)$, such that $\Vol_n(B(x,r_x))= \frac{\Vol_n(M)}{p}$.
By compactness there exists a finite subcollection of balls
$B(x,r_x)$ that cover $M$. By the Vitali covering lemma we can further
choose a subcollection of disjoint balls $B_1, \dots , B_k$ with radii
$r_1, \dots , r_{p'}$, such that balls of three times the radius cover $M$.
Note that we must have $p' \leq p$.
Theorem now follows by taking $U_i = 3 B_i$.

\end{proof}

By Theorem~\ref{thm: main with boundary} we have the following: 
for each open subset  $U$ of $M$ there exists a 
family of cycles $X_t$, for $0 \leq t \leq 1$, sweeping-out $U$. 
Moreover, we have that $X_0$ is a trivial cycle,
$X_1 = \partial U$ 
and $\Vol(X_t) \leq\Vol(\partial U) + C(n)\Vol(U)^{\frac{n-1}{n}}$.
For each $i$ we let $X^i _t$ be the family of cycles with the above properties
for the submanifold with boundary $U_i \setminus (\bigcup_{j=1} ^{i-1} U_j)$.
Let $V_i = \bigcup_{j=1} ^i U_j$ for $1 \leq i \leq p'$ and 
$V_i = \emptyset$ otherwise.
Define a family of mod 2 cycles $Z_t$ for $0 \leq t \leq p'$
by setting
$Z_t= \partial V_{i-1} + X^i_{t-[t]}$ for $i-1 \leq t \leq i$, here $[ t ]$ denotes the integer part of $t$.
We identify the endpoints (which are trivial cycles) and rescale 
so that $Z_t$ is parametrized by a unit circle.

Observe that for each $t$ cycle $Z_t$
can be decomposed into two $(n-1)$-cycles 
$Z_t = Z^{1}_t+Z^{2}_t$ with $Z^{1}_t \subset \bigcup \partial U_i$
and $\Vol (Z^{2}_t) \leq C_4(n) \left( \frac{\Vol(M)}{p}\right)^{\frac{n-1}{n}}$.

Following Gromov \cite{Gromov1988} and Guth \cite{Guth2009} 
we will now define
 a $p$-cycle $F: \RP^p \rightarrow Z_{n-1}(M,\mathbb{Z}_2)$
which detects the cohomology element $\lambda^p$. 
Consider a truncated symmetric product $TP^p(S^1)$.
%i.e. all expressions of the form $\sum_{i=1} ^p a_i t_i$,
%where $a_i \in \mathbb{Z}_2$ and $t_i \in S^1$.
Recall that the symmetric product is defined as a quotient
of $( S^1)^p$ by the symmetric group $S_p$.
The truncated symmetric product is then defined as a quotient
of the symmetric product by an equivalence relation
$1 s_j + 1 s_k=0$ if $s_j=s_k$ and $j \neq k$.
In \cite{Mostovoy1998} Mostovoy proved that $TP^p(S^1)$
is homeomorphic to $\RP^p$ (we learnt about this from \cite{Guth2009}).
We define 
%$F:\RP^p \rightarrow Z(M,\mathbb{Z}_2)$ by 
$F(\sum_{i=1} ^p a_i t_i) = \sum_{i=1} ^p a_i Z_{t_i}$.
Alternatively, we could define a map from $\RP^p$
to $Z_{n-1}(M,\mathbb{Z}_2)$ with the desired property using zeros of 
real polynomials of one variable with degree at most $p$
as in \cite[4.2B]{Gromov1988} and \cite[Section 5]{Guth2009}.

We claim that $\Vol(F(x)) \leq C(n) p^{\frac{1}{n}} \Vol(M)^{\frac{n-1}{n}}$.
Indeed, we may decompose each of the $p$ summands 
$Z_{t_i} = Z^{1}_{t_i}+Z^{2}_{t_i}$ with $Z^{1}_{t_i}$
contained in the union of the boundaries of $U_i$'s 
and the volume of $Z^{2}_{t_i}$ bounded from above by
a constant times $( \frac{\Vol(M)}{p})^{\frac{n-1}{n}}$. 
Since we are dealing with mod 2 cycles the sum of all $ Z^{1}_{t_i}$ can not be greater than 
$\Vol(\bigcup \partial U_i) \leq C(n) p ^ {\frac{1}{n}}\Vol(M)^{\frac{n-1}{n}}$.
We also have that the sum $\sum Z^{2}_{t_i} \leq C(n) p ^ {\frac{1}{n}}\Vol(M)^{\frac{n-1}{n}}$.

Finally, we show that $F^*(\lambda ^p) = F^*(\lambda)^p \neq 0$.  
Observe that $S=\{1 \cdot t : t \in S^1 \} \subset  TP^p(S^1)$
represents a non-trivial homology class in $H_1(TP^p(S^1), \mathbb{Z}_2)$
and $F(S)=\{Z_t \}_{t \in S^1}$ is a sweep-out of $M$
by Proposition \ref{decomposition}.
It follows that $F^*(\lambda)=1$.
This finishes the proof of Theorem \ref{parametric width}.

We use this result to bound volumes of minimal hypersurfaces
in a space with positive Ricci curvature.
These minimal hypersurfaces arise from
Almgren-Pitts min-max theory
as supports of 
stationary almost minimizing integral varifolds.
Pitts~\cite{Pitts1981} and Schoen and Simon~\cite{SchoenSimon}
proved that these hypersurfaces are smooth embedded subamanifolds
when $n \leq 7$. In higher dimensions they may have singular
sets of dimension at most $n-8$.

Marques and Neves  \cite{MarquesNeves2013} showed that
every manifold $M$ of dimension $n$, for $3\leq n \leq 7$, with positive Ricci curvature
possesses infinitely many embedded minimal 
hypersurfaces. 
Theorem \ref{hypersurfaces} is an effective version of their result.
Note that for 2-dimensional surfaces 
an analogous result for periodic geodesics
is false. Morse showed that 
for an ellipsoid of area $1$ with distinct but
very close semiaxes
the length of the fourth shortest geodesic
becomes uncontrollably large (\cite{morse1934}).

%\begin{theorem} \label{MNhypersurfaces}
%Let $M$ be a manifold of positive Ricci curvature
%and $\sys_{n-1}(M)$ denote the volume of a smallest
%minimal hypersurface in $M$. For any $k$ there exists
%$k$ distinct minimal hypersurfaces of volume
%$\leq C'_n  \Vol(M)\left( \sys_{n-1}(M)\right)^{- \frac{1}{n-1}} k ^ %{\frac{1}{n-1}}$.
%\end{theorem}

\vspace{0.1in}

\noindent
\emph{Proof of Theorem \ref{hypersurfaces}.}
Let $\V_k$ be the infimum of numbers such that there exists $k$ distinct minimal
hypersurfaces of volume less or equal to $\V_k$. By \cite[Prop.
4.8]{MarquesNeves2013} and results in Section 2 of the same paper we may assume
that each parametric width $\W^k$ can be written as a finite linear combination
$\W^k = \sum a_j \V_j$, where $a_j$ are integer coefficients.  Moreover, when
$M$ has positive Ricci curvature (or, more generally whenever $M$ has the
property that any two embedded minimal hypersurfaces in $M$ intersect) we have
$\W^k =  a_{j_k} \V_{j_k}$ for some positive integer $a_{j_k}$.

Let $C=C(n)$ be the constant from Theorem \ref{parametric width} and define $C'
= 2^{\frac{1}{n-1}} C ^{\frac{n}{n-1}} $. We proceed by contradiction. Suppose
\[ \V_k > C'  \Vol(M)\left(\sys_{n-1}(M)\right)^{-\frac{1}{n-1}} k^ { \frac{1}{n-1}}\]
for some $k$.  Let $A(N) = \{\W^i \leq N\}$.  It follows from the proof of
Theorem 6.1 in \cite{MarquesNeves2013} that if $\W^i=\W^{i+1}$ for some $i$
then there exists infinitely many hypersurfaces of volume at most $\W^i$.
Hence, we may assume that $\W^i<\W^{i+1}$ for all $i <k$.  By Theorem
\ref{parametric width} we have that the number of elements in the set $A(N)$
satisfies $\#A(N) \geq \frac{N^n}{C^n \Vol(M)^{n-1}}-1$. In particular, we
compute $\#A(\V_k) \geq 2 \frac{k V_k}{V_1} -1$.
On the other hand, the set 
$ \{a_i \V_i : a_i \in \mathbb{N}, a_i \V_i \leq \V_k \}$ 
has at most $\frac{k \V_k}{\V_1}$ elements, 
which is a contradiction.

%\end{proof}

\bibliographystyle{abbrv}
\bibliography{bibliography}

\end{document}